\documentclass[12pt]{amsart}     

\usepackage{amssymb}
\usepackage{amsmath}
\usepackage{amsthm}

\usepackage[T2A]{fontenc}
\usepackage[cp1251]{inputenc}
\newtheorem {theorem} {Theorem}[section]

\newtheorem {corollary} [theorem]{Corollary}
\newtheorem {lemma}  [theorem]{Lemma}
\theoremstyle{definition}

\theoremstyle{remark}
\newtheorem{rem}{Remark}

\newcommand{\lmo}{\mathrm{LMO}}

\newcommand{\rd}{\mathbb{R}^d}

\newcommand{\bmo}{\mathrm{BMO}}

\begin{document}

\title{A T(L) theorem for the
Calder\'{o}n--Zygmund operators in $\bmo(\rd)$}



\author{Andrei V.~Vasin }

%\address{Admiral Makarov State University of Maritime and Inland Shipping,
%Dvinskaya st.~5/7, St.~Petersburg 198035, Russia}
\address{St.~Petersburg Department
of Steklov Mathematical Institute, Fontanka 27, St.~Petersburg 191023, Russia}
\email{andrejvasin@gmail.com}

%\thanks{This research was supported by the Russian Science Foundation (grant No.	23-11-00171).}



\begin{abstract}
 A variant of the global $T(1)$ criterion to characterize the bounded
Calder\'{o}n--Zygmund operators on $\bmo(\rd)$ is proved. We apply it to   the certain  Calder\'on commutators.
\end{abstract}

\maketitle

\section{Introduction}\label{s_int}

%\cite{Bra10, Hy10}

\subsection{Background }
It is well known that    Calder\'on--Zygmund operators  bounded in  $L^2(\rd)$, also  are bounded  from $L^\infty(\rd)$ to $\bmo(\rd)$.
To extend  Calder\'on--Zygmund operators on homogeneous $\bmo(\rd)$ and to have  an estimate on seminorms
\begin{equation}\label{e_first}
   \|Tf\|_{\bmo}\leq C \|f\|_{\bmo},
\end{equation}
 it is assumed   necessarily,  that  $T1= const$ (for convolution operators see \cite{P}).
One may escape this assumption  and  consider \textit{normed} $\bmo$ on bounded domains \cite{V}. This may be  done  even in non doubling setting for the Tolsa space $\textrm{RBMO}(\mu)$ with   finite measure  \cite{DV}.
  Closely related  approach concerns  the Hermite--Calder\'on--Zygmund operators on the Hermite $\bmo_H(\rd)$ space  \cite{B}.

   In this paper we consider the classical space $\bmo(\rd)$ as a normed space, where we define   the  Calder\'on--Zygmund operators.  A criterion  to characterize the bounded Calder\'on--Zygmund operators is proved. Formally, the result obtained is similar to ones in \cite{B,V} and is related to characterization of pointwise multipliers in $\bmo(\rd)$    \cite{NY}.  We apply T(L) theorem to research the certain Calder\'on commutators on
   $\bmo(\rd)$.
\subsection{The space $\bmo(\rd)$ }
Let $Q=Q(c,\ell)$  be a cube in $\rd$ with sides parallel to the axes, with side length $\ell=\ell(Q)$ and the centre $c=c(Q)$, and let $|Q|$ be its volume.
The space $\bmo(\rd)$ consists of those $f\in L^1_{loc}(\rd)$ that have the finite seminorm
\begin{equation}\label{e_df_osc}
\|f\|_{\bmo}=\sup_{Q\subset\rd} \inf_{b_Q \in \mathbb{R}}\frac{1}{|Q|}
\int_Q |f - b_Q| dx.
\end{equation}
One may naturally define a norm in $\bmo(\rd)$. Let $f_Q=\frac{1}{|Q|}\int_{Q}f(x)dx$ be a standard average over a cube $Q$.
 Fix a cube $Q_0$ with the centre in origin and with side length equaled to 1.
Then the  space $\bmo(\rd)$, equipped with the norm
\begin{equation}\label{e_df_osc1}
\|f\|^*_{\bmo}=\|f\|_{\bmo} + |f_{Q_0}|,
\end{equation}
is a Banach space. Replacing  $Q_0$ by any other cube, we obtain a norm equivalent to initial one.

\subsection{Calder\'{o}n--Zygmund operators}

A Calder\'{o}n--Zygmund kernel   is a measurable function
$K (x, y)$ on $\rd\times \rd\setminus\{(x, x): x\in \rd\}$
satisfying the following conditions:
\begin{align}
|K(x, y|
\le C |x-y|^{-d} \label{e_cz1}\\
|K(x_1, y)- K(x_2, y)|
\le C\frac{|x_1- x_2|^\delta}{|x_1- y|^{d+\delta}},\label{e_cz3}
\end{align}
where $ 2|x_1- x_2|\le |x_1- y|$ and $\delta$,
$0 < \delta\le 1$, is a regularity constant specific to the kernel $K$ and $|x-y|$ is  euclidean distance in $\rd$.

A singular  integral operator defined  as
\begin{equation}\label{e_cz}
  Tf(x)=\int K (x, y)f(y)dy.
\end{equation}
when $x\in \rd \backslash\textrm{supp} f$ is called a  Calder\'{o}n-Zygmund operator  associated to the kernel $K (x, y)$ if   it  is extended to  a bounded operator in $L^2(\rd)$.

  Even $T$ is bounded on $L^2(\rd)$, in Section 2 we follow to the modified argument \cite[p. 156-157]{S} to indicate   what does it mean that  $T$ is defined and bounded  in $\bmo(\rd)$.

 \subsection{T(L) theorem and applications}
 To formulate the main result we introduce an auxilliary space.  Given a cube $Q=Q(c,\ell)$,
let $\widetilde{Q}=\widetilde{Q}(\tilde{c}, \tilde{\ell})$ be a smallest cube   that contains a cube $Q$ and the fixed cube $Q_0$. Define a  function (indeed, a log--distance  function between cubes $Q$ and $Q_0$, similar to ones in \cite{Jo,NY})
 \[
 L(Q)=\log \max(\tilde{\ell}/\ell,  \tilde{\ell})+1 .
 \]

The logarithmic  mean oscillation space $\lmo(\rd)$ consists of those $f\in\bmo(\rd)$ that have  finite  seminorm
\begin{equation}\label{lmo}
\|f\|_{\lmo}=\sup_{Q\subset\rd} \inf_{b_Q \in \mathbb{R}}\frac{L(Q)}{|Q|}
\int_Q |f - b_Q| dx.
\end{equation}

 Let us denote $1=\chi_{\rd}$ and
  $\log_{s,-s}(x)=  \log |x-s|-\log|x+s|$, $s\in \rd $.
 The  main theorem  may be interpreted as a   T(1) theorem, as well as a testing criterion  on the certain family $\log |x|\cup \{\log_{s,-s}(x)\}_s $ of typical unbounded functions  in $\bmo(\rd)$.
\begin{theorem}\label{t_main}
Let $T$ be a Calder\'on--Zygmund operator associated to a standard kernel $K$ and bounded on $L^2(\rd)$.
The following properties are equivalent:
\begin{enumerate}
  \item[(i)] $T$ is bounded from the normed space $\bmo(\rd)$ to the homogeneous  $\bmo(\rd)$; there is an estimate with a constant independent of $f$
  \begin{equation}\label{e_11}
    \|Tf\|_{\bmo}\leq C \|f\|^*_{\bmo};
  \end{equation}\[\]
  \item[(ii)] $T1 \in \lmo(\rd)$;
  \item[(iii)] $T\log|x|, \,T\log_{s,-s}(x) \in \bmo(\rd)$ and
   \[\|T\log_{s,-s}(x)\|_{\bmo}\leq C \|\log_{s,-s}(x)\|^*_{\bmo},\]
  uniformly with respect to $s\in \rd$.
    \end{enumerate}
\end{theorem}
\begin{rem}
   Observe that even seminorms  $\|Tf\|_{\bmo}$, as well $\|T1\|_{\lmo}$ are correctly  defined, we can not   define $\int_{Q_0}Tf(x)dx$, in  general. Therefore, we can not replace seminorm  $\|Tf\|_{\bmo}$ by norm $\|Tf\|^*_{\bmo}$ in (\ref{e_11}) (a similar  estimate of maximal function in BMO see in \cite[p. 179]{S}, \cite{BVS}). We return to this issue in Section 2.3.
\end{rem}

To illustrate the T(L) theorem, observe that if $T1=const$, and hence $\| T 1\|_{\lmo}=0$, we obtain a standard estimate  (\ref{e_first}) on seminorms.

    A more interesting application of the T(L) theorem  concerns   the Calder\'on commutators $\mathcal{C}_A$ associated with  the kernel
   \[K_A(x,y)=\frac{A(x)-A(y)}{(x-y)^2},\]
   where  $A(x)\in Lip_1(\mathbb{R})$. The question in order is to state, for which $A(x)\in Lip_1(\mathbb{R})$ the  Calder\'on commutator $\mathcal{C_A}$  is bounded on
 $\bmo(\mathbb{R})$. We have the following sufficiency condition.
 \begin{theorem}\label{t_com}
  Let $A(x) \in Lip_1(\mathbb{R})$ such that $A'\in \lmo(\mathbb{R})$, then $\mathcal{C_A}$ is bounded on $\bmo(\mathbb{R})$.
 \end{theorem}
\begin{rem}
  In fact, the condition  $A'\in \lmo(\mathbb{R})\cap L^\infty(\mathbb{R})$ exactly characterises  the  pointwise multipliers on $\bmo(\mathbb{R})$ \cite{NY}, where one can find examples of the certain multipliers suitable for Theorem \ref{t_com}.
\end{rem}
  The assumption $A'\in \lmo(\mathbb{R})$ is sharp. We check this on
     the Calder\'on commutator $\mathcal{C}$ associated with $A(x)=|x|$.   It turns out that   $\mathcal{C}$ is bounded on $L^2(\mathbb{R})$, as well,
    from $L^\infty(\mathbb{R})$ to $\bmo(\mathbb{R})$ \cite[Section 5.17]{S}. However,
 $\mathcal{C}$ is unbounded on $\bmo(\mathbb{R})$.

\subsection{Organization and notation}
Definition of Calder\'{o}n--Zygmund operators on $\bmo(\rd)$ is given in  Section~2.
T(L) theorem is proved in Section 3. Theorems   \ref{t_com} is proved in  Section 4.

As usual, the letter $C$  denotes a constant, which may be different at each
occurrence and which is independent of the relevant variables under consideration.
If  $A/C\leq  B\leq C A,$ then we write $A\approx B$.

\section{Auxiliary results}
\subsection{Calder\'{o}n--Zygmund operators from $L^\infty(\rd)$ to $\bmo(\rd)$}

  For any cube $Q$ with the centre $c$,   and for a bounded function $f$  the following value is well defined
\begin{equation}\label{e_tq}
  T_Qf(x)=\int_{2Q}K(x,y)f(y)dy+ \int_{\rd\setminus 2Q} (K(x,y)-K(c,y))f(y)dy, \quad x\in Q,
\end{equation}
where by $\alpha Q$ we denote a cube $Q$  dilated  $\alpha $ times with respect to its centre.
The first summand on the right is well defined since $T$ is bounded in $L^2(\rd)$, and the second one is well defined by property (5) of the kernel $K$. Define  the oscilation
\[
\|f\|_Q=\inf_{b_Q\in \mathbb{R}}\frac{1}{|Q|}\int_{Q}|f-b_Q|dx.
\]
 Then we have
\[
\|T_Q f\|_Q\leq C \|f\|_{\infty}
\]
 uniformly with respect to $Q$. If $Tf$ is well defined, for instance for $f\in L^2(\rd)$, then
 \begin{equation}\label{e_bmo}
 \|Tf\|_{\bmo}=\sup_Q \|T_Q f\|_Q.
\end{equation}
So, even $L^2(\rd)\cap L^\infty(\rd)$ is not dense in $L^\infty(\rd)$, we may define  seminorm  $\|Tf\|_{\bmo}$ by (\ref{e_bmo}) in any case. Also, we see that operator $T$ so defined, is bounded from $L^\infty(\rd)$ to homogeneous $\bmo(\rd)$.

\subsection{Calder\'{o}n--Zygmund operators on $\bmo(\rd)$}

We extend the construction of Section 2.1 on $\bmo(\rd)$.
Let $f\in \bmo(\rd)$,  and let $f_Q=\frac{1}{|Q|}\int_{Q}f(x)dx$ be the standard average over a cube $Q$.
Consider the following decomposition
 \[f= f_{2Q} + (f-f_{2Q}).\]
  Define
  \[
  T_Qf=T_Qf_{2Q}+ T_Q(f-f_{2Q})
  \]
 by  formula (\ref{e_tq}) separately for each summand.

Since $f_{2Q}$ is a constant, then by Section 2.1, the first summand $T_Qf_{2Q}=f_{2Q}T_Q1$  is well defined. Even seminorm $\|T_Qf_{2Q}\|_Q$   depends obviously on $Q$, it is not issue for us.

For the second summand, we prove a next lemma that ascends to \cite{H};
see also \cite{DV} for non doubling measure and \cite{B}
  for the Hermite--Calder\'on--Zygmund operators.  We give the proof for the sake of completeness.
\begin{lemma}\label{l_23}
There exists an estimate for seminorms
\[
\|T_Q(f-f_{2Q})\|_Q \leq C \|f\|_{\bmo}
\]
with a constant independent of $f$ and $Q$.
\end{lemma}

\begin{proof}
Recall that
\[\begin{aligned}
&T_Q(f-f_{2Q})(x)\\
 &=\int_{2Q}K(x,y)(f(y)-f_{2Q})dy+\int_{\rd\backslash 2Q} (K(x,y)-K(c,y))(f(y)-f_{2Q})dy.
\end{aligned}\]
For the first summand, we have by the H\"older's inequality
\[
\begin{aligned}
I_1
&=  \frac{1}{| Q|}  \int_Q \left| \int_{2Q}K(x,y)(f(y)-f_{2Q})dy\right|dx\\
&\leq \left(\frac{1}{| Q|}  \int_Q  \left|\int_{2Q}K(x,y)(f(y)-f_{2Q})dy\right|^2 dx\right)^{1/2}.
\end{aligned}
\]
 Since $T$ is bounded on $L^2(\rd)$,
\[
I_1\leq C \left(\frac{1}{| 2Q|}  \int_{2Q} |f(y) - f_{2Q}|^2 dy \right)^{\frac{1}{2}},
\]
and then by the John--Nirenberg inequality \cite[Corollary, p. 144 ]{S} for $f\in \bmo(\rd)$,  one has

\begin{equation}
  \leq C \frac{1}{| 2Q|} \int_{2Q} |f(y) - f_{2Q}|dy \leq C\|f\|_{\bmo}.
\end{equation}

For the second summand, by  property (5) of the kernel $K(x,y)$ for $x\in Q$, one has
\[
\begin{aligned}
I_2
&=  \left|\int_{\rd\setminus 2Q} (K(x,y)-K(c,y))(f(y)-f_{2Q})dy \right| \\
&\leq C \int_{\rd\setminus 2Q} \frac{|x- c|^\delta}{|x- y|^{d+\delta}}|f(y)-f_{2Q}| dy.
\end{aligned}
\]
Split the integral in dyadic sum
\[
\begin{aligned}
I_2
& \leq C \sum_{k=1}^\infty
\frac{\ell^\delta}{(2^{k-1}\ell)^{d+\delta}} \int_{2^{k+1}Q \setminus 2^k Q}|f(y) - f_{2Q}|dy \\
&\leq C\sum_{k=1}^\infty \frac{\ell^\delta}{(2^{k}\ell)^{d+\delta}}
\left( \int_{2^{k+1}Q}|f(y) - f_{2^{k+1}Q}|\, dy
+ |2^{k+1}Q| |f_{2Q} - f_{2^{k+1}Q}|\right).
\end{aligned}
\]
Since $|f_{2Q} - f_{2^{k+1}Q}|\leq C k \|f\|_{\bmo}$, we obtain
\[
\begin{aligned}
I_2
&\leq C \sum_{k=1}^\infty \frac{\ell^\delta}{(2^{k}\ell)^{d+\delta}}
 |2^{k+2}Q| (1+k)
 \|f\|_{\bmo}\\
&\leq C \sum_{k=1}^\infty \frac{k}{2^{k\delta}}\|f\|_{\bmo}\leq C \|f\|_{\bmo}.
\end{aligned}\]
Thus, combining with the estimate for $I_1$, we have
\[
\|T_Q(f-f_{2Q})\|_Q  \leq C \|f\|_{\bmo},
\]
and we are done.
\end{proof}
  Lemma \ref{l_23} implies that for each cube $Q$ and for each $f\in \bmo(\rd)$ seminorm  $\|T_Q f\|_Q$ is well defined. We are ready to define $T$ on
$\bmo(\rd)$.
   Put for any $f\in \bmo(\rd)$
  \[
\begin{aligned}
\|Tf\|_{\bmo}
&=\sup_Q \|T_Q f\|_Q \\
&=\sup_Q \|f_{2Q} T_Q1+ T_Q(f-f_{2Q})\|_Q,
\end{aligned}
\]
 if the right  side is finite.
Observe, that for $f\in L^{\infty}(\rd)$  the definition of seminorm $\|Tf\|_{\bmo}$ coincides with (\ref{e_bmo}) from Section 2.1.

\subsection{Normed version of Theorem \ref{t_main}}
Let us return back to the issue left in Remark 1 of Section 1.4. Obtained the correct definition of seminorm $\|Tf\|_{\bmo}$, one can not define  $Tf$ a.e. for $f\in \bmo(\rd)$.  On the other hand,  $T_Qf$  from (\ref{e_tq}) is well defined a.e. for any $f\in \bmo(\rd)$. So far,   for the fixed cube $Q_0$  we write
\[\widetilde{T}f=T_{Q_0}f.\]
Since
  $T_{Q}f-T_{Q_0}f=C(Q, Q_0)$ is a constant  for any cube $Q$,  it holds
\[\|T_Q f\|_Q=\|T_{Q_0} f\|_Q,\]
and therefore, one has equality for seminorms
\[
\|\widetilde{T}f\|_{\bmo} =\|Tf\|_{\bmo}.
\]
Also, replacing  the fixed cube $Q_0$ by any other cube, we obtain the same seminorm.
The advantage of the operator $\widetilde{T}$  compared to $T$ is  that, now we can estimate the mean value
\[
\begin{aligned}
|(\widetilde{T}f)_{Q_0}|
&\leq |f_{2Q_0}|\| T_{Q_0}1\|_{Q_0}+ \|T_{Q_0}(f-f_{2Q_0})\|_{Q_0}\\
&\leq C(|f_{2Q_0}|+\|f\|_{\bmo})
\end{aligned}
\]
for $f\in \bmo(\rd)$.  This follows
\begin{corollary}
  Let $T$ be a Calder\'on--Zygmund operator associated to a standard kernel $K$ and bounded on $L^2(\rd)$.
Then any of the properties of Theorem \ref{t_main} is equivalent to

 $\mathrm{(\tilde{i})}$   $\widetilde{T}$ is bounded on the normed space $\bmo(\rd)$, and it holds
  \begin{equation}
    \|\widetilde{T}f\|^*_{\bmo}\leq C \|f\|^*_{\bmo}
  \end{equation}
  with a constant independent of $f$.
\end{corollary}
  However,  it must be considered that when the both operators  are well defined,  then  $\widetilde{T}$ is a one-dimensional perturbation of  a Calder\'on-Zygmund operator $T$
 \[
  Tf(x)-\widetilde{T}f(x)
 =\int_{\rd\backslash 2Q_0} K(0,y)(f(y)-f_{2Q_0})dy.
 \]

\section{Proof of T(L) theorem}
Implication (i)$\Rightarrow $(iii) is obvious. Let us prove others.
\subsection{Proof (ii)$\Rightarrow$(i)}
By Lemma \ref{l_23}, we see that
\[
\begin{aligned}
\|Tf\|_{\bmo}
&\leq C(\sup_Q |f_{2Q}|\,\| T_Q1\|+ \|f\|_{\bmo})\\
\end{aligned}
\]
Then,
\[|f_{2Q}|\leq |f_{2Q}-f_{\tilde{Q}}|+ |f_{\tilde{Q}}-f_{Q_0}|+|f_{Q_0}|,\]
 where $Q_0$ is a fixed cube and $\tilde{Q}$ is the smallest cube containing $2Q$ and $Q_0$.
 Clearly
 $|f_{2Q}-f_{\tilde{Q}}|\leq C (\log \tilde{\ell}/\ell+1)\|f\|_{\bmo} $
 and
 $|f_{Q_0}-f_{\tilde{Q}}|\leq C (\log \tilde{\ell}+1) \|f\|_{\bmo} $,
therefore,
 \begin{equation}\label{dir_av}
  |f_{2Q}|\leq C L(Q)\|f\|_{\bmo}+|f_{Q_0}|,
 \end{equation}
  and hence,
  \[
\begin{aligned}
\|Tf\|_{\bmo}
&\leq C (\sup_Q (L(Q)\|f\|_{\bmo}+|f_{Q_0}|)\| T_Q 1\|_Q+ \|f\|_{\bmo})\\
&\leq C (\|f\|_{\bmo}\sup_Q L(Q)\| T_Q 1\|_Q +|f_{Q_0}| \sup_Q \| T_Q 1\|_Q + \|f\|_{\bmo}).
\end{aligned}
\]
Note that $\sup_Q L(Q)\| T_Q 1\|_Q=\| T 1\|_{\lmo}$ and $\sup_Q \| T_Q 1\|_Q=\| T 1\|_{\bmo}$, so continue
\[
\begin{aligned}
\|Tf\|_{\bmo}
&\leq C ( \|f\|_{\bmo}\| T 1\|_{\lmo}+ |f_{Q_0}|\| T 1\|_{\bmo}+\|f\|_{\bmo})\\
& \leq C (\| T 1\|_{\lmo}\|f\|_{\bmo}+\| T 1\|_{\bmo} |f_{Q_0}|)\\
&\leq C (\| T 1\|_{\lmo}+\| T 1\|_{\bmo})\|f\|^*_{\bmo}.
\end{aligned}
\]
  Of course,   $\| T 1\|_{\bmo} < \infty$, since $T$ is bounded from $L^\infty(\rd)$ to homogeneous $\bmo(\rd)$.  This completes the proof (ii)$\Rightarrow$(i).

\subsection{Proof (iii)$\Rightarrow$(ii)}
We start with a  lemma.
\begin{lemma}\label{l_ext}
  Let $\log_Q$ be the mean value of the function $\log|x|$  over a cube $Q=Q(c,\ell)$. Then, it holds
 \[|\log_Q- \log (\max (|c|, \ell))|\leq C \]
 with a constant independent of $Q$.
\end{lemma}
\begin{proof}
 Since $\log|x| \in \bmo(\rd)$,   for an arbitrary cube $Q$ there is a constant $b_{Q}$
 such that
 \[ \frac{1}{|Q|}\int_Q |\log |x|  - b_{Q}| dx \leq C \|\log|x| \,\|_{\bmo}.\]
   It turns out that one may take
  $b_Q=\log |c|$, when $|c|\geq \ell$, and $b_Q=\log\ell$,  when $|c|\leq\ell$.
  This fact is  easily derived by scaling argument  \cite[p. 140--141]{S}.

  With the  choice of $b_Q$ we have
  \[  |\log_Q- \log \max (|c|, \ell)|\leq C \|\log|x|\, \|_{\bmo},\]
 that is required.
 \end{proof}
  Proceed to the proof of  $(iii)\Rightarrow (ii)$.

  \textbf{Step 1.}
   By assumptions of $(iii)$,
 \[\|T \log|x| \,\|_{\bmo}\leq C \|\log|x|\,\|^*_{\bmo}\approx 1,
\]
and hence, by Lemma \ref{l_23} and
by triangle inequality,  we have uniformly with respect to  $Q$
\[ |\log_{2Q}|\| T_Q1\|_Q\leq C  .\]
This follows
\[ (|\log_{2Q}|+1)\| T_Q1\|_Q\leq C+  \| T_Q1\|_Q.\]
 By Lemma \ref{l_ext},
 \[
  |\log (\max (|c|, \ell))|\leq |\log_{2Q}|+C
 \]
 with a constant independent of $Q$, and therefore,
 \[
|\log (\max (|c|, \ell))|\| T_Q1\|_Q \leq C +  \| T_Q1\|_Q,
\]
  If  $|c|\leq \max(\ell, 1)$, i.e. when the euclidean  distance
 between cubes $Q$ and $Q_0$ is small enough compared to the size of the cubes, then  it holds
 \[
  L(Q)\approx |\log (\max (|c|, \ell))| + 1.
 \]
  So,   for  cubes of such kind we obtained an estimate
  \begin{equation}\label{e_last}
   L(Q)  \| T_Q1\|_Q \leq C( 1+  \| T_Q1\|_Q),
\end{equation}
where a constant is independent of $Q$.

\textbf{Step 2.}  Observe, that if $|c|\geq \max(\ell, 1)$, i.e.   when euclidean  distance
 between cubes $Q$ and $Q_0$ is large enough compared to the size of the cubes, then
 \[L(Q)\approx \log\frac{|c|}{\ell}+ 1\]
  and therefore $L(Q)$ is not majorized by $|\log (\max (|c|, \ell))|+ 1$,  in general.
 However, we  prove the estimate (\ref{e_last}) for cubes of this kind. Let $\log_s |x|=\log|x-s|$ be a shift    and  $\log_{s,Q}$ be the mean value of it over a cube $Q$.
For
 the  family of functions $\log_{s,-s}(x)=  \log_s|x|-\log_{-s}|x|$, $s\in \rd $ we prove a lemma.
 \begin{lemma}\label{l_last}
  Norms  $\|\log_{s,-s}(x)\|^*_{\bmo}$ are bounded uniformly in  $s\in\rd$.
 \end{lemma}
 \begin{proof}
  Obviously, seminorms of  $\log_{s,-s}(x)$ are estimated  uniformly in $s$,  as follows
  \[
\begin{aligned}
\|\log_{s,-s}(x)\|_{\bmo}
&\leq 2\|\log|x|\|_{\bmo}\leq C.
\end{aligned}
\]
    Also, by Lemma \ref{l_ext},  we have
   \[|\log_{s,Q}-\log\max(|c-s|,\,\ell)|\leq C .\]
   Therefore,  for the mean value $\log_{s,-s,Q}$ over $Q$ of   $\log_{s,-s}(x)$ we get
 \begin{equation}\label{e_llast}
 |\log_{s,-s,Q} -(\log\max(|c-s|,\,\ell)-\log\max(|c+s|,\,\ell))|    \leq C  .
 \end{equation}
  Particularly, for the  cube $Q_0$  with $c=0$ and $\ell=1$
     \[  |\log_{s,-s,Q_0} |\leq C,\]
     where $C$ is independent of $s$.
Combining, we have
 \[
\begin{aligned}
\|\log_{s,-s}(x)\|^*_{\bmo}
&= \|\log_{s,-s}(x)\|_{\bmo}+|\log_{s,-s,Q_0} |\leq C ,
\end{aligned}
\]
that is required.
 \end{proof}
We are ready to finish the proof.
  By assumptions of $(iii)$ and  Lemma \ref{l_last},
\[
\|T (\log_{s,-s}(x)) \|_{\bmo} \leq C.
\]
Then, by Lemma \ref{l_23} and by triangle inequality,
  it holds
  \[
|\log_{s,-s,2Q}|\, \|T_Q1\|_Q   \leq C
 \]
and, so,
\[
(|\log_{s,-s,2Q}|+1)\, \|T_Q1\|_Q   \leq C+\|T_Q1\|_Q .
 \]
By (\ref{e_llast}), we have
\[
|\log\max(|c-s|,\,2\ell)-\log\max(|c+s|,\,2\ell)|\,\| T_Q1\|_Q
\leq C  +  \|T_Q1\|_Q.
 \]
Put $s=-c$, we obtain
\[|\log\max(|c|/\ell,\,1)| \,\| T_Q1\|_Q \leq C +\|T_Q1\|_Q.\]
Since $L(Q)\approx |\log(\max(|c|/\ell,\,1))|+1$, when $|c|\geq \max(\ell, 1)$, one has
\[
L(Q) \,\| T_Q1\|_Q
\leq C (1+\|T_Q1\|_Q)\\
\]
that gives us (\ref{e_last})
for the rest of the cubes.

Now, taking the  supremum  with respect to  $Q$ in (\ref{e_last}), we get
 \[
 \|T1\|_{\lmo}\leq C(1 +\|T1\|_{\bmo}).
 \]
 Since $T$ is bounded from $L^\infty(\rd)$ to homogeneous $\bmo(\rd)$,   $\|T1\|_{\bmo}<\infty$.  This completes the proof (iii)$\Rightarrow$(ii).

\section{Calder\'on commutators}

\subsection{Proof of Theorem \ref{t_com}}
  We will prove  $\mathcal{C}_{A}1\in \lmo(\mathrm{R})$, and apply Theorem \ref{t_main}. 
  
  For arbitrary interval $Q=(c-\ell/2, c+\ell/2)$ let $2Q=(c-\ell, c+\ell)$ be a dilated interval with the same centre.
 Let $\mathcal{C}_{A,Q}$ be the operators we get by  (\ref{e_tq}) for $\mathcal{C}_A$.  
 We start with  the proof that $\mathcal{C}_{A,Q}$
inherits the certain properties of  $\mathcal{C}_{A}$, indicated in \cite[p. 312]{S}.
 \begin{lemma}
 For $x\in Q$ one has the equality
\[\mathcal{C}_{A,Q}1(x)=\mathcal{H}_Q A'(x)+C(Q),\]
where $\mathcal{H}$ is the Hilbert transform with the kernel $\frac{1}{x-y}$ and $C(Q)$ is a constant.
 \end{lemma}
 \begin{proof}
   We have
 \[
 \begin{aligned}
\mathcal{C}_{A,Q}1(x)
=\int_{2Q}\frac{A(x)-A(y)}{(x-y)^2}dy+ \int_{\rd\setminus 2Q} \left(\frac{A(x)-A(y)}{(x-y)^2}-\frac{A(c)-A(y)}{(c-y)^2}\right)dy,
\end{aligned}
 \]
where the first term  is a PV integral. Integrating by parts
\[
 \begin{aligned}
\mathcal{C}_{A,Q}1(x)
&=\frac{A(x)-A(y)}{x-y}\bigg|^{c+\ell}_{y=c-\ell}+  \left(\frac{A(x)-A(y)}{x-y}-\frac{A(c)-A(y)}{c-y}\right) \bigg|^{c-\ell}_{y=-\infty}+\bigg|^{\infty}_{y=c+\ell}\\
&+\int_{2Q}\frac{A'(y)}{x-y}dy+ \int_{\rd\setminus 2Q} \left(\frac{A'(y)}{x-y}-\frac{A'(c)}{c-y}\right)dy\\
&=N(x)+\mathcal{H}_Q A'(x).
\end{aligned}
 \]
The nonintegral term $N(x)$ is a constant on $Q$. Indeed,
\[
N(x)
=\frac{A(c)-A(y)}{c-y}\bigg|^{c+\ell}_{y=c-\ell} +  \left(\frac{A(x)-A(y)}{x-y}-\frac{A(c)-A(y)}{c-y}\right) \bigg|^{\infty}_{y=-\infty}.
 \]
Since $A' \in L^\infty(\mathbb{R})$ implies  $A(y)=o(y^2)$ at infinity,  the second summand above is zero, while the first one depends on $c$ and $\ell$. That is
$ N(x)=C(Q)$.
 \end{proof}

 Even  the  classical Hilbert transform is zero on constants,  it takes some times in  setting of the paper to check  $\mathcal{H}_Q1(x)=0$ for   $x\in Q$. Indeed,
 \[
 \begin{aligned}
\mathcal{H}_{Q}1(x)
&=\int_{2Q}\frac{1}{x-y}dy+ \int_{\rd\setminus 2Q} \left(\frac{1}{x-y}-\frac{1}{c-y}\right)dy\\
&=-\log|x-y|\bigg|^{c+\ell}_{c-\ell}- (\log|x-y|-\log|c-y|)\bigg|^{c-\ell}_{-\infty} +\bigg|^{\infty}_{c+\ell}\\
&=\log|x-c|\bigg|^{c-\ell}_{c+\ell}=0.
\end{aligned}
 \]
Hence, for $x\in Q$ we may write
 \[\mathcal{C}_{A,Q}1(x)=\mathcal{H}_Q (A'-A'_{2Q})(x)+C(Q),\]
where $A'_{2Q}$ denotes the mean value with respect to $2Q$. So,
\[ \begin{aligned}
\|\mathcal{C}_{A,Q}1\|_Q
&=\|\mathcal{H}_Q (A'-A'_{2Q})\|_Q.\\
\end{aligned}
\]
In order to estimate $\|\mathcal{H}_Q (A'-A'_{2Q})\|_Q$, consider  the terms  $I'_1$ and $I'_2$ we get  from the terms $I_1$ and $I_2$ in  Lemma \ref{l_23}, replacing $f$ by $A'$ and $T$ by $\mathcal{H}$. Repeating argument of Lemma \ref{l_23},  it holds for $I'_1$
 \[
\begin{aligned}
 I'_1
 &=\frac{1}{| Q|}  \int_Q \left| \int_{2Q}\frac{A'(y)-A'_{2Q}}{x-y}dy\right|dx\\
&\leq C \frac{1}{| 2Q|} \int_{2Q} |A'(y)-A'_{2Q}|dy\\
&= C\frac{L(Q)}{L(Q)} \frac{1}{|2Q|} \int_{2Q} |A'(y)-A'_{2Q}|dy\\
& \leq C \frac{1}{L(Q)}\|A'\|_{\lmo}.
\end{aligned}
\]
Taking care of the term $I'_2$ is not so easy. Split it in dyadic sum,
 \[
\begin{aligned}
I'_2
&= \left|\int_{\mathbb{R}\setminus 2Q} (A'(y)-A'_{2Q})\left(\frac{1}{x-y}-\frac{1}{c-y}dy \right)\right| \\
&\leq C\sum_{k=1}^\infty \frac{\ell^\delta}{(2^{k}\ell)^{1+\delta}}
\left( \int_{2^{k+1}Q\backslash2^k Q}|A'(y) - A'_{2^{k+1}Q}|\, dy+ |2^{k+1}Q| |A'_{2Q} - A'_{2^{k+1}Q}|\right),
\end{aligned}
\]
where by  $A'_{2^kQ}$ we define  the mean values of $A'$ over $2^kQ$, By telescopic summation we have
\[
\begin{aligned}
|A'_{2Q} - A'_{2^{k+1}Q}|
& \leq \sum_{s=1}^k |A'_{2^sQ} - A'_{2^{s+1}Q}|\\
 & \leq C \sum_{s=1}^k \frac{1}{|2^{k+1}Q|} \int_{2^{s+1}Q}|A'(y) - A'_{2^{s+1}Q}|dy,
  \end{aligned}
\]
and substituting in $I'_2$, it holds
\[
\begin{aligned}
I'_2
&\leq C\sum_{k=1}^\infty \frac{\ell^\delta}{(2^{k}\ell)^{1+\delta}}\bigg( \int_{2^{k+1}Q}|A'(y) - A'_{2^{k+1}Q}| dy\\
& + |2^{k+1}Q| \sum_{s=1}^k \frac{1}{|2^{s+1}Q|} \int_{2^{s+1}Q}|A'(y) - A'_{2^{k+1}Q}|dy \bigg)\\
& \leq C\sum_{k=1}^\infty \frac{1}{2^{k\delta}}\sum_{s=1}^k \frac{1}{|2^{s+1}Q|} \int_{2^{s+1}Q}|A'(y) - A'_{2^{k+1}Q}|dy \\
& =C\sum_{k=1}^\infty \frac{1}{2^{k\delta}}\sum_{s=1}^k \frac{L(2^{s+1}Q)}{L(2^{s+1}Q)|2^{s+1}Q|} \int_{2^{s+1}Q}|A'(y) - A'_{2^{k+1}Q}|dy \\
& \leq C\sum_{k=1}^\infty \frac{1}{2^{k\delta}}\sum_{s=1}^k \frac{\|A'\|_{\lmo}}{L(2^{s+1}Q)}\\
& \approx \sum_{s=1}^\infty \frac{1}{2^{s\delta}L(2^{s+1}Q)}\|A'\|_{\lmo},
\end{aligned}
\]
where in the last line we change the summation order.

We claim that
\begin{equation}\label{100}
  \sum_{s=1}^\infty \frac{1}{2^{s\delta}L(2^{s+1}Q)}\leq C \frac{1}{L(Q)}.
\end{equation}
With  estimate (\ref{100}) in hand, we get
   \[
\begin{aligned}
\|\mathcal{C}_{A,Q}1\|_Q
&=\|\mathcal{H}_Q A'\|_Q\\
&\leq C ( I'_1+I'_2)\\
&\leq C \frac{1}{L(Q)}\|A'\|_{\lmo},
\end{aligned}
\]
which follows  $\mathcal{C}_{A}1\in \lmo(\mathrm{R})$, and Theorem \ref{t_main} completes the proof.
 
 So, it remains to prove (\ref{100}). Consider
 two cases according to whether Euclidean distance between a cube $Q$ and the fixed cube $Q_0$ is less or greater 1.

\textbf{Case 1.}  If the  distance $d$ between a cube $Q$ and the fixed cube $Q_0$ is less than 1, we compare $\ell_s=\ell(2^sQ)=2^s\ell$
with 1.
 If $\ell_s\leq1$, that is $s\leq\log 1/\ell$, then $ L(2^sQ)\approx \log  1/\ell_s+1$, while if
$\ell_s>1$, i.e. $s>\log 1/\ell$, then $ L(2^sQ)\approx \log  \ell_s+1$. Also, split the sum (\ref{100}) on level $\ell_s=1$, then  if the first sum on the
right hand side is non empty, it holds
\[
\begin{aligned}
 \sum_{s=1}^\infty \frac{1}{2^{s\delta}L(2^{s+1}Q)}
 &=\sum_{\ell_s\leq1} \frac{1}{2^{s\delta}(\log  1/\ell_s+1)}+ \sum_{\ell_s>1}\frac{1}{2^{s\delta}(\log  \ell_s+1)}\\
&=\sum_{s\leq\log 1/\ell} \frac{1}{2^{s\delta}(\log  1/2^s\ell+1)}+\sum_{s>\log 1/\ell} \frac{1}{2^{s\delta}(\log  2^s\ell+1)}\\
 &\leq C \int_{1}^{1/\ell} \frac{dt}{t^{1+\delta}(\log  1/t\ell+1)}+\int_{1/\ell}^{\infty} \frac{dt}{t^{1+\delta}(\log  t\ell+1)}\\
&\leq C\frac{1}{\log  1/\ell+1}+ \ell^{\delta}\\
&\leq C \frac{1}{\log  1/\ell+1}\\
&\leq C\frac{1}{L(Q)}.
\end{aligned}
\]
Obviously, the same estimate is valid if the first sum is empty (for instance if $\ell> 1$). Then the   sum (\ref{100}) is estimated   by  $\ell^{\delta}= o(1/L(Q))$.

\textbf{Case 2.}  If the  distance $d$ between a cube $Q$ and the fixed cube $Q_0$ is not less than 1, we compare $\ell_s$ with $d$.
We have $ L(2^sQ)\approx \log  (d/\ell_s+1)$, if $\ell_s\leq d$, and $ L(2^sQ)\approx \log  (\ell_s+1)$, if
$\ell_s>d$. Also, split the sum \ref{100} on level $\ell_s=d$. Like the previous observation, whether  the first sum on the right side is  empty or not, it holds

\[
\begin{aligned}
 \sum_{s=1}^\infty \frac{1}{2^{s\delta}L(2^{s+1}Q)}
 &=\sum_{\ell_s\leq d} \frac{1}{2^{s\delta}(\log  d/\ell_s+1)}+ \sum_{\ell_s> d}\frac{1}{2^{s\delta}(\log  \ell_s+1)}\\
&=\sum_{s\leq\log d/\ell} \frac{1}{2^{s\delta}(\log  d/2^s\ell+1)}+\sum_{s>\log d/\ell} \frac{1}{2^{s\delta}(\log  2^s\ell+1)}\\
 &\leq C \int_{1}^{d/\ell} \frac{dt}{t^{1+\delta}(\log  d/t\ell+1)}+\int_{d/\ell}^{\infty} \frac{dt}{t^{1+\delta}(\log  t\ell+1)}\\
&\leq C \frac{1}{\log  d/\ell+1}+ (\ell/d)^{\delta}\\
&\leq C \frac{1}{\log  d/\ell+1}\\
&\leq C\frac{1}{L(Q)},
\end{aligned}
\]
that is required  for   the  second case and for Theorem \ref{t_com}.

\subsection {Sharpness of condition $A'\in \lmo(\mathbb{R})$ }
 Let $\mathcal{C}$ be the Calder\'on commutator associated to  $A(x)=|x|$. Let $Q=(c/2, 3c/2)$ be an interval with the centre $c>0$ and length $|Q|=c$, 
We will prove that $ \|\mathcal{C}_Q1\|_Q$ is separated from zero. From  (\ref{lmo}) this  follows
 $\mathcal{C}1\notin \lmo (\mathbb{R})$, and by Theorem \ref{t_main}, $\mathcal{C}$ is unbounded in $\bmo(\mathbb{R}$).

  Defining $\mathcal{C}_Q$ by (\ref{e_tq}),   we write for $y\in Q$ 
 \[
 \begin{aligned}
\mathcal{C}_Q1(x)
&=\int_{2Q}\frac{|x|-|y|}{(x-y)^2}dy+ \int_{\rd\setminus 2Q} \left(\frac{|x|-|y|}{(x-y)^2}-\frac{|c|-|y|}{(c-y)^2}\right)dy\\
&=PV\int_{0<y<2c}\frac{1}{x-y}dy\\
&+\int_{y>2c} \left(\frac{1}{x-y}-\frac{1}{c-y}\right)dy\\
&+\int_{y<0} \left(\frac{x+y}{(x-y)^2}-\frac{c+y}{(c-y)^2}\right)dy\\
&=2(\log x-\log c).
\end{aligned}
 \]
 Then,
\[
 \begin{aligned}
 \|\mathcal{C}_Q1\|_Q
&=\inf_{b_Q \in \mathbb{R}}\frac{2}{|Q|}\int_Q|\log x-\log c-b_Q|dy\\
&=2\inf_{b_Q \in \mathbb{R}} \int_{1/2}^{3/2}|\log t-b_Q|dt\\
&\geq C>0,
\end{aligned}
\]
 that is required.
\bibliographystyle{amsplain}
%\bibliography{T1rbmo}

\end{document}